\documentclass[12pt]{amsart}
\usepackage{epsfig,color}
\usepackage{blindtext}
\usepackage{hyperref}
\usepackage{graphicx}
\usepackage{enumitem} 
\usepackage{url}
\usepackage{amssymb}
\usepackage{graphicx,import}
\usepackage{comment}
\usepackage{mathrsfs}

\theoremstyle{definition}

\newtheorem{theorem}{Theorem}[section]
\newtheorem{definition}[theorem]{Definition}

\newtheorem{lemma}[theorem]{Lemma}
\newtheorem{remark}[theorem]{Remark}
\newtheorem{corollary}[theorem]{Corollary}

\newtheorem*{remark*}{Remark}

\numberwithin{equation}{section}

\newcommand{\mb}{\mathbb}

\newcommand{\R}{\mathbb{R}}

\newcommand{\cH}{\mathcal H}

\title[On the Entropy of Parabolic Allen-Cahn Equation]{On the Entropy of Parabolic Allen-Cahn Equation}

\date{\today}

\author{Ao Sun}
\address{Department of Mathematics, University of Chicago
	5734 S. University Avenue, 
	Chicago, IL 60637, USA}
\email{aosun@uchicago.edu}

\begin{document}
\maketitle

\begin{abstract}
	We define a (mean curvature flow) entropy for Radon measures in $\R^n$ or in a compact manifold. Moreover, we prove a monotonicity formula of the entropy of the measures associated with the parabolic Allen-Cahn equations. If the ambient manifold is a compact manifold with non-negative sectional curvature and parallel Ricci curvature, this is a consequence of a new monotonicity formula for the parabolic Allen-Cahn equation. As an application, we show that when the entropy of the initial data is small enough (less than twice of the energy of the one-dimensional standing wave), the limit measure of the parabolic Allen-Cahn equation has unit density for all future time. 
	
	
\end{abstract}

\section{Introduction}
The parabolic Allen-Cahn equation
\begin{equation}
	\frac{\partial}{\partial t}u^\epsilon=\Delta u^\epsilon-\frac{1}{\epsilon^2}f(u^\epsilon)
\end{equation}
was introduced by Allen and Cahn in 1979. It is the gradient flow of the energy functional
\begin{equation}
	M^\epsilon(u)=\int_{\R^n}\frac{\epsilon}{2}|Du|^2+\frac{1}{\epsilon}F(u)dx
\end{equation}
with a speed up factor $1/\epsilon$. Here $F(u)$ is the potential function known to be the ``double well potential", and $f$ is the derivative of the potential function, see Section 2.

The solutions to the Allen-Cahn equation are the models to the motion of phase boundaries by surface tension. When $\epsilon\to 0$, the term $-\frac{1}{\epsilon^2}f(u^\epsilon)$ makes the phase boundry sharp, and the limit should be the motion of surfaces by mean curvature. In \cite{Ilmanen93}, Ilmanen studied the measures $\mu_t^\epsilon$ associated with the solution $u^\epsilon$, which is defined by
\begin{equation}
	d\mu_t^\epsilon=\left(\frac{\epsilon}{2}|Du^\epsilon(\cdot,t)|^2+\frac{1}{\epsilon}F(u^\epsilon(\cdot,t))\right)dx.
\end{equation}

In this paper, we study the (mean curvature flow) entropy of this measure. Entropy in $\R^n$ was introduced by Colding-Minicozzi \cite{Colding-Minicozzi12} in the study of mean curvature flow. Later the author \cite{sun2019entropy} generalized Colding-Minicozzi's idea to define entropy in a manifold. Entropy is a quantity that characterizes a submanifold/measure from all scales. Hypersurfaces with small entropy have been studied in many contexts, see Colding-Ilmanen-Minicozzi-White \cite{CIMW13}, Bernstein-Wang \cite{BW16},\cite{Bernstein-Wang18} and Zhu \cite{Zhu16}.

The monotonicity formula of the entropy of mean curvature flow plays a very important role in the study of mean curvature flow, see \cite{Colding-Minicozzi12}, \cite{sun2019entropy}. The main result of this paper is the following monotonicity formula of the entropy of Allen-Cahn equations. We use $\lambda$ to denote the entropy (see Definition \ref{Def:entropy}), and we use $\xi^\epsilon$ to denote the discrepancy measure (see \eqref{eq:discrepancy}). 

\begin{theorem}\label{thm:monotonicity of entropy-epsilon}
	Suppose $M=\R^n$ or $M$ is a compact manifold with non-negative sectional curvature and parallel Ricci curvature. Suppose $\mu^\epsilon_t$ is the measure associated with the parabolic Allen-Cahn equation on $M$. Suppose $\xi^\epsilon_0\leq0$. For every $0\leq t_1\leq t_2$, we have
	\begin{equation}
		\lambda(\mu^\epsilon_{t_2})\leq\lambda(\mu^\epsilon_{t_1}).
	\end{equation}
\end{theorem}

If we let $\epsilon\to 0$, we obtain a monotonicity formula of entropy for the limit measure.

\begin{theorem}\label{thm:monotonicity of entropy}
	Suppose $M=\R^n$ or $M$ is a compact manifold with non-negative sectional curvature and parallel Ricci curvature. Suppose $\mu^\epsilon_t$ is the measure associated with the parabolic Allen-Cahn equation on $M$, and $\mu_t$ is the limit measure as $\epsilon\to 0$. For every $0<t_1\leq t_2$, we have
	\begin{equation}
		\lambda(\mu_{t_2})\leq\lambda(\mu_{t_1}).
	\end{equation}
\end{theorem}

We remark that there are also monotonicity formulas for local entropy (see Definition \ref{Def:local entropy}), and Theorem \ref{thm:monotonicity of entropy-epsilon} and Theorem \ref{thm:monotonicity of entropy} are just special cases. See Corollary \ref{Cor:monotonicity of entropy-epsilon} and Corollary \ref{Cor:monotonicity of entropy}.

The monotonicity of entropy provides information on the evolutions. In this paper, we focus on the problem of the density of the measures associated with the parabolic Allen-Cahn equations. In \cite{Ilmanen93}, Ilmanen proved that under certain reasonable initial conditions, as $\epsilon\to 0$, $\mu_t^\epsilon$ converge to a rectifiable measure $\mu_t$. See also \cite{ESS},\cite{Soner97}. Moreover $\mu_t$ is a Brakke flow, a geometric measure theoretic weak solution of mean curvature flow. In \cite[Section 13]{Ilmanen93}, Ilmanen asked the following question: when does $\mu_t$ have unit density? Let $\alpha$ be the energy of the $1$-dimensional standing wave (see (\ref{Eq:alpha})). Then unit density means that $\mu_t$ has density $\alpha$ almost everywhere. 

In this paper, we use entropy and local entropy of Radon measures (see Definition \ref{Def:entropy}) to answer this question in a special case. 

\begin{theorem}\label{Thm:low entropy imply unit density}
	Suppose $M=\R^n$ or $M$ is a compact manifold with non-negative sectional curvature and parallel Ricci curvature. Suppose $\mu^\epsilon_t$ are measures associated with the Allen-Cahn equations on $M$. Suppose there exists $\kappa>0$ such that $\lambda^{(0,T)}(\mu^\epsilon_0)<2\alpha-\kappa$, and the discrepancy measure $\xi_0^\epsilon\leq 0$ for a sequence of $\epsilon\to 0$. Then $\mu^\epsilon_t$ has unit density for $t\in(0,T)$. Here $T\in(0,\infty)$ or $T=\infty$. 
\end{theorem}

In the above theorem, we need a technical assumption that the discrepancy measure is non-positive. This assumption is used by Ilmanen in \cite{Ilmanen93}. Even we do not require the discrepancy measure to be non-positive, we can still use an argument by Soner \cite{Soner97} to show that the limit measure has unit density.

\begin{theorem}\label{Thm:Main Theorem}
	Suppose $M=\R^n$ or $M$ is a compact manifold with non-negative sectional curvature and parallel Ricci curvature. Suppose $\mu^\epsilon_t$ are measures associated with the Allen-Cahn equations on $M$, and $\mu_t$ is a limit measure as $\epsilon\to 0$. Suppose there exists $\kappa>0,\epsilon_0>0$ and $\delta>0$ such that the entropy $\lambda(\mu_\delta^\epsilon)<2\alpha-\kappa$, or the local entropy $\lambda^{(0,T)}(\mu_\delta^\epsilon)<2\alpha-\kappa$, for $\epsilon<\epsilon_0$. Then the limit measure $\mu_t$ has unit density for $t\geq \delta$, i.e. the density of $\mu_t$ is $\alpha$ almost everywhere.
\end{theorem}

Bronsard-Stoth \cite{Bronsard-Stoth96} constructed examples such that the limits do not have unit density. The limit of the initial data constructed by Bronsard-Stoth has density $2\alpha$, this implies that the entropy of their initial data is at least $2\alpha$ (see Lemma \ref{Lem:density is bounded by entropy}). Thus our theorem is sharp in some sense. We highlight that our theorems provide long time information, i.e. they hold for arbitrarily large time $t$.

There is a connection between the unit density problem of the limit of the Allen-Cahn equation and the unit density problem of Brakke flow. In \cite[Appendix E]{ilmanen1994elliptic}, Ilmanen asked that can we show the unit density is preserved for the Brakke flow starting from the boundaries of sets. Ilmanen also constructed an example called ``double spoon" to illustrate that in general unit density may fail under Brakke flow, even when the initial measure has unit density. Thus the assumption that the Brakke flow starts from the boundaries of sets is necessary.

It is not too hard to construct a motion of the boundaries of sets by the limit of the Allen-Cahn equation, see \cite[Section 1.4]{Ilmanen93}. In particular, the double spoon example can not be obtained by the Allen-Cahn equation. So the unit density conjecture on Brakke flow is true if the unit density problem for the limit of the Allen-Cahn equation is true. As a special case, our result implies that for a Brakke flow, if the entropy of the initial measure is small (less than $2$), the unit density is preserved under the flow. 

\subsection*{Idea of proof} The proof relies on monotonicity formulas of Allen-Cahn equation (see Theorem \ref{thm:monotonicity of F-epsilon in M} and Theorem \ref{thm:monotonicity of F in M}). If the ambient manifold is $\R^n$, then the monotonicity formula has been proved by Ilmanen \cite{Ilmanen93} and Soner \cite{Soner97}, and similar results for the non-scalar case (correspondence to higher codimensional mean curvature flow) was studied by Ambrosio-Soner in \cite{AS98} and Jerrard-Soner in \cite{JS02}. We refer the readers to \cite{Soner95} for an overview discussion.

In this paper, we generalize the monotonicity formula to a compact manifold $M$ with non-negative sectional curvature and parallel Ricci curvature. The proof follows the calculations by Ilmanen and Soner, with an extra term appearing because the ambient space $M$ is not $\R^n$. In order to handle this extra term, we need an idea by Hamilton. In \cite{Hamilton93M}, Hamilton proved a monotonicity formula for mean curvature flow in a compact manifold $M$. Similarly, because the ambient space $M$ is not $\R^n$, there is an extra term, and Hamilton used a Harnack inequality \cite{Hamilton93H} to bound this term. In our case, we can also use Hamilton's Harnack inequality to bound the extra term.

The curvature assumption on $M$ is necessary for Hamilton's Harnack inequality. Without the curvature assumption, we can only obtain a Gronwall type inequality rather than a monotonicity formula, and the entropy may become very large as $t\to\infty$. Thus the curvature assumption is necessary if we want to study long time behaviour. 

We want to remark that Pisante-Punzo \cite{Pisante-Punzo} proved a Huisken's type of inequality for parabolic Allen-Cahn equations in a manifold, without a specific curvature assumption. Their inequality was motivated by Ecker's local monotonicity formula, see \cite{Ecker}. Nevertheless, their inequality is a Gronwall type inequality, so it can not be used to study long time behaviours.

\section{Preliminaries on Allen-Cahn Equation}\label{S:Preliminaries}

Let $u^\epsilon$ be the unique smooth solutions of the equation
\begin{equation}
	\begin{split}
		\frac{\partial}{\partial t}u^\epsilon=\Delta u^\epsilon-\frac{1}{\epsilon^2}f(u^\epsilon)&\text{ on $\R^n\times[0,\infty)$}\\
		u^\epsilon(\cdot,0)=u_0^\epsilon(\cdot)& \text{ on $\R^n\times\{0\}$}.
	\end{split}
\end{equation}
where $u_0^\epsilon$ is the initial data. The potential function $F:\R\to\R$ satisfies
\[f=F',\quad F=\frac{1}{2}g^2,\]
where 
\begin{equation}\label{Eq:potential function}
	\left\{\begin{array}{l}
		f(-1)=f(0)=f(1)=0,\\
		f>0 \text{ on } (-1,0),\ \ f<0 \text{ on } (0,1),\\
		f'(-1)>0,\ \ f'(1)>0,\ \ f'(0)<0\\
		g(-1)=g(1)=0,\ \ g>0\text{ on } (-1,1).
	\end{array}\right.
\end{equation}

In this paper, we will assume $F$, $f$ and $g$ are the functions in the standard model 
\begin{equation}
	F(u)=\frac{1}{2}(1-u^2)^2,\quad f(u)=2u(u^2-1),\quad g(u)=1-u^2.
\end{equation}
For general $F$, $f$ and $g$ satisfying \eqref{Eq:potential function}, all the discussions in this paper still hold if we assume one extra assumption on $F$ (see \eqref{eq:x^2FleqC}). We will discuss this assumption in Remark \ref{rmk:generalF}.

The one-dimensional standing wave $q^\epsilon$ is defined to be the solution to
\begin{equation}
	q^\epsilon_{xx}-\frac{1}{\epsilon^2}f(q^\epsilon(x))=0,\quad x\in\R
\end{equation}
with the assumption that $q_x^\epsilon>0$, $q^\epsilon(\pm\infty)=\pm1$ and $q^\epsilon(0)=0$. We can solve this ODE by solving the first order ODE
\begin{equation}
	q^\epsilon_{x}-\frac{1}{\epsilon}g(q^\epsilon(x))=0,\quad x\in\R
\end{equation} 
with the assumption that $q_x^\epsilon>0$, $q^\epsilon(\pm\infty)=\pm1$ and $q^\epsilon(0)=0$.

By ODE theory we can solve $q^\epsilon$ for any potential function satisfying (\ref{Eq:potential function}). We define $\alpha$ to be its energy
\begin{equation}\label{Eq:alpha}
	\alpha=\int_{-\infty}^\infty\frac{\epsilon}{2}(q_x^\epsilon(x))^2+\frac{1}{\epsilon}F(q^\epsilon(x)) dx.
\end{equation}
By change of variable we know that $\alpha$ is independent of $\epsilon$. One can also check that
\[\alpha=\int_{-1}^1 \sqrt{F(s)/2}ds.\]

In particular, for the model case $F(u)=\frac{1}{2}(1-u^2)^2$, we have $q^\epsilon(x)=\tanh(x/\epsilon)$, $\alpha=4/3$.

We define the Radon measure $\mu_t^\epsilon$ by 
\begin{equation}
	d\mu_t^\epsilon=\left(\frac{\epsilon}{2}|Du^\epsilon(\cdot,t)|^2+\frac{1}{\epsilon}F(u^\epsilon(\cdot,t))\right)dx.
\end{equation}
We say $\mu_t^\epsilon$ is associated with the solution to the Allen-Cahn equation, or for the sake of brevity say $\mu_t^\epsilon$ is associated with the Allen-Cahn equation.

Ilmanen \cite{Ilmanen93} proved that, under certain technical requirement, there exists $\epsilon_i\to 0$ such that $\mu_t^{\epsilon_i}$ converge to a $(n-1)$-rectifiable Radon measure $\mu_t$ for a.e. $t>0$. Moreover, $\mu_t$ is a mean curvature flow in the sense of Brakke. 

One motivation of Ilmanen to study the limit behavior of the Allen-Cahn equation was to understand the weak mean curvature flow, see \cite[Section 12]{Ilmanen93}. Thus Ilmanen had some technical requirement on the initial data $\mu_0^\epsilon$, see \cite[p.423]{Ilmanen93}. However, uniformly bounded total energy $\int d\mu_t^\epsilon$ is enough to give a convergence subsequence, see \cite{Soner97}. \cite{Soner97} even proved that a weaker assumption is enough to yield a convergence subsequence. In our case we only require the initial data have uniform small entropy. This can be viewed as a restriction on (iv) on \cite[p.423]{Ilmanen93}, cf. Lemma \ref{Lem:equivalence of entropy and area growth}.

\bigskip

Next we discuss the \emph{discrepancy measure $\xi_t$}, which is defined to be 
\begin{equation}\label{eq:discrepancy}
	d\xi_t^\epsilon=\left(\frac{\epsilon}{2}|Du^\epsilon(\cdot,t)|^2-\frac{1}{\epsilon}F(u^\epsilon(\cdot,t))\right)dx.
\end{equation}

If we define $r^\epsilon$ to be the function satisfying
\begin{equation}
	u^\epsilon=q^\epsilon(r^\epsilon),
\end{equation}
then $r^\epsilon$ satisfies the equation
\begin{equation}\label{eq:evolution of r}
	\frac{\partial}{\partial t}r^\epsilon=\Delta r_\epsilon+\frac{2g'}{\epsilon}(|D r^\epsilon|^2-1).
\end{equation}
We can check
\begin{equation}
	|Dr^\epsilon|^2 = \frac{(\epsilon/2)|Du^\epsilon|^2}{(1/\epsilon)F(u^\epsilon)}.
\end{equation}
This implies that $|Dr^\epsilon|^2-1$ carries information of $\xi_t$. In particular, $|Dr^\epsilon|\leq 1$ implies that $\xi_t\leq 0$. 

By using \eqref{eq:evolution of r} and the equation of $|Dr^\epsilon|^2$, Ilmanen and Soner obtained some estimates on $|D r^\epsilon|^2$. As a result, they proved certain bounds on $\xi_t$. We refer the readers to \cite[Section 4]{Ilmanen93} \cite[Appendix]{Soner97} for detailed discussions. Here we only state their results which we will use later. We also remark that although their results were proved for the Allen-Cahn equations in $\R^n$, the proofs are also valid for the Allen-Cahn equations in a closed manifold (in fact the maximum principle is even easier in this case because the domain is compact). 

\begin{theorem}[\cite{Ilmanen93}, Section 4]\label{thm:discrepancy negative}
	If $\xi_0\leq 0$, then $\xi_t\leq 0$ for all $t>0$.
\end{theorem}

\begin{theorem}[\cite{Soner97}, Proposition 4.1]\label{thm:discrepancy estimate}
	Assume $|u^\epsilon(\cdot,0)|\leq 1$. Then there exists $0<\epsilon_0<1$ such that for any $\epsilon\leq \epsilon_0$ we have
	\begin{equation}
		|Dr^\epsilon(x,t)|^2\leq 1+\frac{2}{\log(1/\epsilon)}\frac{(\epsilon r^\epsilon(x,t))^2+1}{t}.
	\end{equation}
	for $(x,t)\in M\times(0,\infty)$.
\end{theorem}

We remark that In \cite{Soner97}, Soner used $z^\epsilon$ to denote $\epsilon r^\epsilon$.
\section{Monotonicity formula}
Based on Huisken's monotonicity formula \cite{Huisken90}, Ilmanen \cite[Section 3]{Ilmanen93} proved a monotonicity formula for the measures which are associated with the solutions to the Allen-Cahn equation, with a technical assumption on the initial data (see Ilmanen \cite[1.4 (i)]{Ilmanen93}). Later Soner \cite{Soner97} removed this assumption. Let us first state the monotonicity formula for Allen-Cahn equations on $\R^n$. Let $\rho_{y,s}$ be the backward heat kernel in $\R^n$ multiplying by $\sqrt{4\pi(s-t)}$:
\[\rho_{y,s}(x,t)=\frac{1}{(4\pi(s-t))^{(n-1)/2}}e^{-\frac{|x-y|^2}{4(s-t)}}.\]

\begin{theorem}(\cite[3.3]{Ilmanen93})\label{thm:monotonicity of F-epsilon in R^n}
	Suppose $\mu^\epsilon_t$ is the measure associated with an Allen-Cahn equation on $\R^n$. Suppose the discrepancy measure $\xi_0$ at time $0$ is non-positive. Then for every $0\leq t_1\leq t_2<s$, we have
	\begin{equation}\label{Eq:monotonicity of F-epsilon}
		\int \rho_{y,s}(t_2,x)d\mu^\epsilon_{t_2}(x)\leq \int\rho_{y,s}(t_1,x)d\mu^\epsilon_{t_1}(x).
	\end{equation}
\end{theorem}

\begin{theorem}(\cite[Corollary 5.1]{Soner97})\label{thm:monotonicity of F in R^n}
	Suppose $\mu_t$ is the limit measure associated with Allen-Cahn equations on $\R^n$. For every $0<t_1\leq t_2<s$, we have
	\begin{equation}\label{Eq:monotonicity of F}
		\int \rho_{y,s}(t_2,x)d\mu_{t_2}(x)\leq \int\rho_{y,s}(t_1,x)d\mu_{t_1}(x).
	\end{equation}
\end{theorem}

In this section, we are going to generalize this monotonicity formulas to the parabolic Allen-Cahn equation on a certain class of manifolds. Suppose $M$ is a manifold. We use $\cH(x,y,t)$ to denote the heat kernels on $M$ and we call $\cH(x,y,s-t)$ the backward heat kernel if we fix $s>0$. We have
\[\partial_t\cH(x,y,t)=\Delta_x\cH(x,y,t).\]
Moreover we define
\begin{equation}
	\rho_{y,s}(t,x)=\sqrt{4\pi(s-t)}\cH(x,y,s-t).
\end{equation}
Note that if $M=\R^n$, then this definition coincides with the definition of $\rho_{y,s}$ in $\R^n$ as in the beginning of this section.

In the rest of this section we prove the following monotonicity formulas.

\begin{theorem}\label{thm:monotonicity of F-epsilon in M}
	Suppose $\mu^\epsilon_t$ is the measure associated with Allen-Cahn equations on $M$, where $M$ is a compact manifold with parallel Ricci curvature and non-negative sectional curvature. Suppose the discrepancy measure $\xi_0$ at time $0$ is non-positive. For every $0\leq t_1\leq t_2<s$, we have
	\begin{equation}
		\int \rho_{y,s}(t_2,x)d\mu^\epsilon_{t_2}(x)\leq \int\rho_{y,s}(t_1,x)d\mu^\epsilon_{t_1}(x).
	\end{equation}
\end{theorem}

\begin{theorem}\label{thm:monotonicity of F in M}
	Suppose $\mu_t$ is the limit measure associated to Allen-Cahn equations on $M$, where $M$ is a compact manifold with parallel Ricci curvature and non-negative sectional curvature. For every $0<t_1\leq t_2<s$, we have
	\begin{equation}
		\int \rho_{y,s}(t_2,x)d\mu_{t_2}(x)\leq \int\rho_{y,s}(t_1,x)d\mu_{t_1}(x).
	\end{equation}
\end{theorem}

\begin{proof}[Proof of Theorem \ref{thm:monotonicity of F-epsilon in M} and Theorem \ref{thm:monotonicity of F in M}]
	
	The proof follows the calculations of \cite{Ilmanen93},\cite{Soner97}, and an idea of Hamilton in \cite{Hamilton93M}. In the following, for two matrices $A$ and $B$, $A:B$ denotes the inner product of them as matrices. We also define the unit normal vector to the level sets by $\nu=\frac{Du^\epsilon}{|Du^\epsilon|}$. Note that $\nu$ is not well-defined at the points where $Du^\epsilon=0$, but it does not matter the later calculations, and we can define $\nu$ to be any unit vector at the point where $Du=0$.
	
	Recall the calculations in \cite[3.2]{Ilmanen93}: for $\phi\in C^2(M,R^+)$,
	\begin{equation}
		\begin{split}
			\frac{d}{dt}\int\phi d\mu_t^\epsilon
			=&
			-\epsilon\int\phi\left(-\Delta u+\frac{1}{\epsilon^2}f(u)-\frac{Du\cdot D\phi}{\phi}\right)^2 dx
			\\
			&+
			\int\left(\Delta \phi +\frac{\partial}{\partial t}\phi\right)
			d\mu_t^\epsilon
			+
			\int\left(-\nu\otimes \nu:D^2\phi+\frac{(\nu\cdot D\phi)^2}{\phi}\right)
			(d\mu_t^\epsilon+d\xi_t^\epsilon)
			\\
			\leq &
			\int\left(\Delta \phi +\frac{\partial}{\partial t}\phi\right)
			d\mu_t^\epsilon
			+
			\int\left(-\nu\otimes \nu:D^2\phi+\frac{(\nu\cdot D\phi)^2}{\phi}\right)
			(d\mu_t^\epsilon+d\xi_t^\epsilon)
		\end{split}
	\end{equation}
	Now we insert $\rho_{y,s}(x,t)$ as the test function $\phi$. Note that
	\[\frac{\partial}{\partial t} \rho_{y,s}(x,t)=-\Delta \rho_{y,s}(x,t)-\frac{1}{2(s-t)}\rho_{y,s}(x,t).\]
	Moreover, by Hamilton's Hanarck inequality (see \cite[Corollary 4.4]{Hamilton93H}),
	\[\nu\otimes\nu:\left(D^2\cH-\frac{D\cH\otimes D \cH}{\cH}+\frac{1}{2(s-t)}\cH g\right)\geq 0,\]
	while $\cH=\cH(x,y,s-t)$ and $g$ is the metric tensor on $M$. These imply
	
	\begin{equation}\label{eq:last step of dt int rho}
		\begin{split}
			\frac{d}{dt}\int\rho_{y,s} d\mu_t^\epsilon
			\leq
			&
			\int\left(\Delta \rho_{y,s} +\frac{\partial}{\partial t}\rho_{y,s}\right)
			d\mu_t^\epsilon
			+
			\int\frac{1}{2(s-t)}\rho_{y,s}
			(d\mu_t^\epsilon+d\xi_t^\epsilon)
			\\
			\leq& 
			\int\frac{1}{2(s-t)}\rho_{y,s}d\xi_t^\epsilon.
		\end{split}
	\end{equation}
	
	If $\xi_0\leq 0$, then Theorem \ref{thm:discrepancy negative} shows that $\xi_t\leq 0$, so we complete the proof of Theorem \ref{thm:monotonicity of F-epsilon in M}. In the following we are going to prove Theorem \ref{thm:monotonicity of F in M}. Namely, even the discrepancy measure is not assumed to be non-positive, the inequality
	\begin{equation}
		\int \rho_{y,s}(t_2,x)d\mu_{t_2}(x)\leq \int\rho_{y,s}(t_1,x)d\mu_{t_1}(x)
	\end{equation}
	holds for the limit measure $\mu_t$.
	
	By Theorem \ref{thm:discrepancy estimate}, \eqref{eq:last step of dt int rho} becomes
	
	\begin{equation}
		\frac{d}{dt}\int\rho_{y,s} d\mu_t^\epsilon
		\leq
		\int\frac{1}{2(s-t)}\frac{1}{\epsilon}\frac{2}{\log(1/\epsilon)}\frac{(\epsilon r^\epsilon)^2+1}{t}F(u^\epsilon)\rho_{y,s}dx.
	\end{equation}
	
	We note that 
	\begin{equation}
		\frac{1}{\epsilon}F(u^\epsilon)dx\leq d\mu_t^\epsilon
	\end{equation}
	and
	\begin{equation}\label{eq:x^2FleqC}
		x^2 F(q^\epsilon(x))\leq 4,\qquad \forall x\in\R.
	\end{equation}
	
	So we have
	\begin{equation}
		\begin{split}
			\frac{d}{dt}\int\rho_{y,s} d\mu_t^\epsilon
			\leq&
			\frac{1}{(s-t)\log(1/\epsilon)t}\int \rho_{y,s}d\mu_t^\epsilon
			+
			\frac{4\epsilon}{(s-t)\log(1/\epsilon)t}\int \rho_{y,s} dx,
		\end{split}
	\end{equation}
	
	and when $t\geq t_1$,
	
	\begin{equation}
		\begin{split}
			\frac{d}{dt}\int\rho_{y,s} d\mu_t^\epsilon
			\leq&
			\frac{1}{(s-t)\log(1/\epsilon)t_1}\int \rho_{y,s}d\mu_t^\epsilon
			+
			\frac{4\epsilon}{(s-t)\log(1/\epsilon)t_1}\int \rho_{y,s} dx.
		\end{split}
	\end{equation}
	
	We also note that
	\[
	\int\rho_{y,s}(t,x)dx =\sqrt{4\pi(s-t)}\int \cH(x,y,s-t)dx=\sqrt{4\pi(s-t)}.
	\]
	
	In conclusion, we obtain the following differential inequality 
	\begin{equation}
		\frac{d}{dt}\int\rho_{y,s} d\mu_t^\epsilon
		\leq
		\frac{1}{(s-t)\log(1/\epsilon)t_1}\left(\int \rho_{y,s}d\mu_t^\epsilon
		+
		4\epsilon\sqrt{4\pi(s-t)}
		\right).
	\end{equation}
	
	Then we obtain a Gronwall type inequality for $t\geq t_1$
	\begin{equation}
		\frac{d}{dt}\left(
		\left(\frac{s-t}{s-t_1}\right)^K\int\rho_{y,s}d\mu_{t}^\epsilon
		-8\epsilon\sqrt{\pi} K\int_{t_1}^t (s-b)^{-1/2} \left(\frac{s-b}{s-t_1}\right)^K db\right)
		\leq 0,
	\end{equation}
	where $K=(\log(1/\epsilon)t_1)^{-1}$. Therefore,
	
	\begin{equation}
		\int\rho_{y,s}d\mu_{t_2}^\epsilon
		\leq
		\left(\frac{s-t_1}{s-t_2}\right)^K\int\rho_{y,s}d\mu_{t_1}^\epsilon
		+8\epsilon\sqrt{\pi} K\int_{t_1}^{t_2} (s-b)^{-1/2} \left(\frac{s-b}{s-t_2}\right)^K db.
	\end{equation}
	By letting $\epsilon\to 0$, we obtain the desired monotonicity formula in Theorem \ref{thm:monotonicity of F in M}.
\end{proof}

\begin{remark}\label{rmk:generalF}
	In the proof we assume 
	\[
	F(u)=\frac{1}{2}(1-u^2)^2,\quad f(u)=2u(u^2-1),\quad g(u)=1-u^2.
	\]
	For a general triple $F$, $f$ and $g$ satisfying \eqref{Eq:potential function}, the proof still works if \eqref{eq:x^2FleqC} holds (the right hand side can be replaced by any fixed positive number). 
\end{remark}

\section{Entropy of Parabolic Allen-Cahn Equation}
We follow the idea of Colding-Minicozzi \cite{Colding-Minicozzi12} to define the entropy of a Radon measure, c.f.\cite{sun2019entropy}.
\begin{definition}\label{Def:entropy}
	Suppose $M=\R^n$ or $M$ is a compact manifold. Given a Radon measure $\mu$ in $M$, we define the \emph{entropy} $\lambda(\mu)$ to be
	\[\lambda(\mu)=\sup_{(y,s)\in M\times(0,\infty)}\int\rho_{y,s}(x,0)d\mu(x).\]
\end{definition}

Entropy is a quantity that characterizes the measure from all scales. For example, the following lemma indicates that the entropy is equivalent to the volume growth bound of a measure:

\begin{lemma}\label{Lem:equivalence of entropy and area growth}
	Suppose $M=\R^n$ or $M$ is a compact manifold with non-negative Ricci curvature. There exists $C>0$ only depending on the geometry of $M$ such that for any Radon measure $\mu$ with bounded entropy, 
	\begin{equation}
		C^{-1}\sup_{x\in M,R>0}\frac{\mu(B_R(x))}{R^{n-1}}\leq\lambda(\mu)\leq C\sup_{x\in M,R>0}\frac{\mu(B_R(x))}{R^{n-1}}.
	\end{equation}
\end{lemma}

Here we only prove the case that $M=\R^n$. When $M$ is a compact manifold with non-negative Ricci curvature, the proof is very similar but needs some analysis on heat kernels, which is out of the scope of this paper, so we omit the proof here and refer the readers to \cite[Section 4]{sun2019entropy} for a proof. 

\begin{proof}[Proof of Lemma \ref{Lem:equivalence of entropy and area growth} when $M=\R^n$]
	On one hand, for any $x\in\mb R^{n}$ and $R>0$, we have
	\[\lambda(\mu)\geq \frac{1}{(4\pi t)^{n/2}}\int e^{\frac{-|y-x|^2}{4t}}\chi_{B_R(x)}d\mu(y) \geq \frac{1}{(4\pi t)^{n/2}}e^{-R^2/4t}\mu(B_R(x)).\]
	Then by choosing $t=R^2$ we have $\frac{\mu(B_R(x))}{R^n}\leq C\lambda(\mu)$, where $C$ is a constant.
	
	On the other hand, we only need to prove $\int e^{-|x|^2}d\mu(x)\leq C \sup_{x\in\mb R^{n},R>0}\frac{\mu(B_R(x))}{R^{n-1}}$ to conclude the second inequality. Let $\chi_{B_r(x)}$ be the characteristic function of the set $B_r(x)$.
	\begin{equation}
		\begin{split}
			\int e^{-|x|^2}d\mu(x)&\leq\sum_{y\in\mb Z^{n}}\int e^{-|x|^2}\chi_{B_2(y)}d\mu
			\leq C\sum_{y\in\mb Z^{n}} e^{-|y|^2}\mu(B_2(y))\\
			&\leq C\sup_{x\in\mb R^{n+1},R>0}\frac{\mu(B_R(x))}{R^{n-1}} \sum_{y\in\mb Z^{n}}e^{-|y|^2}\leq C \sup_{x\in\mb R^{n+1},R>0}\frac{\mu(B_R(x))}{R^{n-1}}.
		\end{split}
	\end{equation}
	Here $C$ varies from line to line, but it is always a constant only depending on $n$. $\mb Z^n$ consists of all the integer points in $\R^n$. Then we conclude this Lemma.
\end{proof}

If we only take the supremum over a subset of $\R^n\times(0,\infty)$, we get a localized version of entropy, see \cite{sun2018local}.

\begin{definition}\label{Def:local entropy}
	Given a Radon measure $\mu$ and $U\subset M$, $I\subset(0,\infty)$, we define the \emph{local entropy} $\lambda^I_U(\mu)$ to be
	\[\lambda^I_U(\mu)=\sup_{(y,s)\in U\times I}\int\rho_{y,s}(x,0)d\mu(x).\]
	We will omit the subscription $U$ if $U=M$.
\end{definition}

If we only study the local property, the entropy/local entropy actually gives a bound on the density. We define the $(n-1)$-dimensional density $\theta(x)$ by
\begin{equation}
	\theta(x)=\lim_{r\to 0}\frac{\mu(B_r(x))}{\omega_{n-1}r^{n-1}}
\end{equation}
whenever the limit exists. Here $\omega_{n-1}$ is the volume of the unit ball in $\R^{n-1}$.

\begin{lemma}\label{Lem:density is bounded by entropy}
	Suppose $M=\R^n$ or $M$ is a compact manifold. Suppose $\mu$ is an integral $(n-1)$-rectifiable Radon measure in $M$, then for $T\in(0,\infty)$ or $T=\infty$,
	\begin{equation}
		\theta(x)\leq \lambda^{(0,T)}(\mu).
	\end{equation}
\end{lemma}

\begin{proof}
	First let us consider the case that $M=\R^n$. Since $\mu$ is integral rectifiable, $\theta(x)$ just counts the multiplicity of the approximate tangent plane through $x$. Note the integral of the backward heat kernel on a hyperplane is $1$. Thus
	\[\theta(x)=\lim_{s\to 0}\int\frac{1}{(4\pi s)^{(n-1)/2}}e^{-\frac{|y-x|^2}{4s}}d\mu(y)\leq \lambda^{(0,T)}(\mu).\]
	
	If $M$ is a compact manifold, then the short time expansion of the heat kernels together with a similar argument as the case $M=\R^n$ show the same inequality.
\end{proof}

Now we focus on the measures associated with the parabolic Allen-Cahn equations. Taking supremum of the monotonicity formula among all $(y,s)$ leads to the monotonicity of entropy and local entropy:

\begin{corollary}\label{Cor:monotonicity of entropy-epsilon}
	Suppose $M=\R^n$ or $M$ is a compact manifold with non-negative sectional curvature and parallel Ricci curvature. Suppose $\mu^\epsilon_t$ is the measure associated with the parabolic Allen-Cahn equation on $M$. Suppose $\xi^\epsilon0\leq0$. For every $0\leq t_1\leq t_2$, we have
	\begin{equation}
		\lambda(\mu^\epsilon_{t_2})\leq\lambda(\mu^\epsilon_{t_1}).
	\end{equation}
	
	More generally, given $T>0$, we have
	\begin{equation}
		\lambda^{(0,T)}(\mu^\epsilon_{t_2})\leq\lambda^{(0,T+(t_2-t_1))}(\mu^\epsilon_{t_1}).
	\end{equation}
\end{corollary}

\begin{corollary}\label{Cor:monotonicity of entropy}
	Suppose $M=\R^n$ or $M$ is a compact manifold with non-negative sectional curvature and parallel Ricci curvature. Suppose $\mu^\epsilon_t$ is the measure associated with the parabolic Allen-Cahn equation on $M$, and $\mu_t$ is the limit as $\epsilon\to 0$. For every $0<t_1\leq t_2$, we have
	\begin{equation}
		\lambda(\mu_{t_2})\leq\lambda(\mu_{t_1}).
	\end{equation}
	
	More generally, given $T>0$, we have
	\begin{equation}
		\lambda^{(0,T)}(\mu_{t_2})\leq\lambda^{(0,T+(t_2-t_1))}(\mu_{t_1}).
	\end{equation}
\end{corollary}

\begin{proof}
	Note $\rho_{y,s}(t,x)=\rho_{y,s-t}(0,x)$. Taking supremum among all $(y,s)\in\R^n\times(t_2,\infty)$ on the left hand side of \eqref{Eq:monotonicity of F-epsilon} and \eqref{Eq:monotonicity of F} gives us the monotonicity of entropy, and taking supremum among all $(y,s)\in\R^n\times(t_2,t_2+T)$ on the left hand side of \eqref{Eq:monotonicity of F-epsilon} and \eqref{Eq:monotonicity of F} gives us the monotonicity of local entropy.
\end{proof}

Together with Lemma \ref{Lem:equivalence of entropy and area growth}, Corollary \ref{Cor:monotonicity of entropy-epsilon} and Corollary \ref{Cor:monotonicity of entropy} imply the following uniform volume growth bound.

\begin{corollary}
	Suppose $M=\R^n$ or $M$ is a compact manifold with non-negative sectional curvature and parallel Ricci curvature. Suppose $\mu^\epsilon_t$ is the measure associated with the parabolic Allen-Cahn equation on $M$, and $\mu_t$ is the limit as $\epsilon\to 0$. Then
	\[\sup_{x\in M,R>0}\frac{\mu_t(B_R(x))}{R^{n-1}}\leq C \sup_{x\in M,R>0}\frac{\mu_\delta(B_R(x))}{R^{n-1}},\quad \forall t\geq \delta>0.\]
	If we further assume $\xi^\epsilon_0\leq 0$, then
	\[\sup_{x\in M,R>0}\frac{\mu^\epsilon_t(B_R(x))}{R^{n-1}}\leq C \sup_{x\in M,R>0}\frac{\mu^\epsilon_0(B_R(x))}{R^{n-1}}.\]
	Here $C$ is a constant depending on $M$ but does not depend on time $t$.
\end{corollary}

\section{Unit Density of Limit Measure}
Let $\mu_t$ be the limit of the measures which is associated with the parabolic Allen-Cahn equation. Recall that 

\[\alpha=\int_{-1}^1 \sqrt{F(s)/2}ds,\]

and we say $\mu_t$ has unit density if $\theta(x)=\alpha$ for almost all time and almost all $x$ in the support of $\mu_t$. Note $\alpha$ appears because of the nature of the Allen-Cahn equation. In this section we prove the main theorem.

\begin{proof}[Proof of Theorem \ref{Thm:low entropy imply unit density}]
	By Corollary \ref{Cor:monotonicity of entropy}, $\lambda^{(0,T-t)}(\mu_t)\leq \lambda^{(0,T)}(\mu_0)<2\alpha$.   Then Lemma \ref{Lem:density is bounded by entropy} implies that the density of $\mu_t$ is strictly less than $2\alpha$. Tonegawa \cite{Tonegawa03} proved that $\alpha^{-1}\mu_t$ is integral, thus $\mu_t$ has density $\alpha$, i.e. $\mu_t$ has unit density.
\end{proof}

\begin{proof}[Proof of Theorem \ref{Thm:Main Theorem}]
	By the lower semi-continuity of entropy and local entropy, $\lambda^{(0,T)}(\mu_0^\epsilon)<2\alpha-\kappa$ implies that $\lambda^{(0,T)}(\mu_0)\leq 2\alpha-\kappa<2\alpha$. Then Theorem \ref{Thm:low entropy imply unit density} implies that $\mu_t$ has unit density when $t\in[0,T)$. 
\end{proof}
\subsection*{Acknowledgments}
I want to thank Professor Bill Minicozzi and Christos Mantoulidis for helpful conversations.


\begin{thebibliography}{SK}
	
	
	
	
	\normalsize
	\baselineskip=17pt
	
	
	\bibitem[AS]{AS98}
	\textsc{Ambrosio, L., Soner, H. M.}
	A measure-theoretic approach to higher codimension mean curvature flows.
	\emph{Ann. Scuola Norm. Sup. Pisa Cl. Sci. (4)} \textbf{25} (1997), 27–49.
	
	\bibitem[BW1]{BW16} 
	\textsc{Bernstein, J., Wang, L.}
	A sharp lower bound for the entropy of closed hypersurfaces up to dimension six.
	\emph{Invent. Math.} \textbf{206} (2016), no. 3, 601-627.
	
	
	\bibitem[BW2]{Bernstein-Wang18} 
	\textsc{Bernstein, J., Wang, L.}
	Topology of closed hypersurfaces of small entropy.
	\emph{Geom. Topol.} \textbf{22} (2018), no. 2, 1109-1141.
	
	
	\bibitem[BS]{Bronsard-Stoth96} 
	\textsc{Bronsard, L., Stoth, B.}
	Ginzburg-{L}andau equation and motion by mean curvature. {I}. {C}onvergence.
	\emph{Math. Res. Lett.} \textbf{3} (1996), no. 1, 41-50.
	
	\bibitem[CIMW]{CIMW13} 
	\textsc{Colding, T. H., Ilmanen, T., Minicozzi, II, W. P., White, B.}
	The round sphere minimizes entropy among closed self-shrinkers.
	\emph{J. Differential Geom.} \textbf{95} (2013), no. 1, 53-69.
	
	\bibitem[CM]{Colding-Minicozzi12} 
	\textsc{Colding, T. H., Minicozzi, II, W. P.}
	Generic mean curvature flow {I}: generic singularities.
	\emph{Ann. of Math. (2)} \textbf{175} (2012), no. 2, 755-833.
	
	\bibitem[E]{Ecker} 
	\textsc{Ecker, K.}
	A local monotonicity formula for mean curvature flow.
	\emph{Ann. of Math. (2)}  \textbf{154} (2001), no. 2, 503-525.
	
	\bibitem[ESS]{ESS} 
	\textsc{Evans, L. C., Soner, H. M., Souganidis, P. E.}
	Phase transitions and generalized motion by mean curvature.
	\emph{Comm. Pure Appl. Math.}  \textbf{45} (1992), no. 9, 1097-1123.
	
	\bibitem[Ha1]{Hamilton93M} 
	\textsc{Hamilton, R. S.}
	A matrix {H}arnack estimate for the heat equation.
	\emph{Comm. Anal. Geom.} \textbf{1} (1993), no. 1, 113-126.
	
	\bibitem[Ha2]{Hamilton93H} 
	\textsc{Huisken, G.}
	Asymptotic behavior for singularities of the mean curvature flow.
	\emph{J. Differential Geom.} \textbf{31} (1990), no. 1, 285-299.
	
	\bibitem[Hu]{Huisken90} 
	\textsc{Hamilton, R. S.}
	A matrix {H}arnack estimate for the heat equation.
	\emph{Comm. Anal. Geom.} \textbf{1} (1993), no. 1, 113-126.
	
	\bibitem[HT]{Hutchinson-Tonegawa00} 
	\textsc{Hutchinson, J. E., Tonegawa, Y.}
	Convergence of phase interfaces in the van der {W}aals-{C}ahn-{H}illiard theory.
	\emph{Calc. Var. Partial Differential Equations} \textbf{10} (2000), no. 1, 49-84.
	
	\bibitem[I1]{Ilmanen93}
	\textsc{Ilmanen, T.}
	Convergence of the Allen-Cahn equation to Brakke's motion by mean curvature.
	\emph{J. Differential Geom.} \textbf{38} (1993), no. 2, 417-461.
	
	\bibitem[I2]{ilmanen1994elliptic}
	\textsc{Ilmanen, T.}
	Elliptic regularization and partial regularity for motion by mean curvature.
	\emph{Mem. Amer. Math. Soc.} \textbf{108} (1994), no. 520, x+90 pp.
	
	\bibitem[JS]{JS02}
	\textsc{Jerrard, R. L., Soner, H. M.}
	Limiting behavior of the Ginzburg-Landau functional.
	\emph{J. Funct. Anal.} \textbf{192} (2002), no. 2, 524–561.
	
	\bibitem[PP]{Pisante-Punzo} 
	\textsc{Pisante, A., Punzo, F.}
	Allen-{C}ahn approximation of mean curvature flow in {R}iemannian manifolds {I}, uniform estimates.
	\emph{Ann. Sc. Norm. Super. Pisa Cl. Sci. (5)} \textbf{15} (2016), 309-341.
	
	
	\bibitem[So1]{Soner95} 
	\textsc{Soner, H. M.}
	Front propagation. Boundaries, interfaces, and transitions (Banff, AB, 1995).
	\emph{CRM Proc. Lecture Notes} \textbf{13}, 185–206.
	
	\bibitem[So2]{Soner97} 
	\textsc{Soner, H. M.}
	Ginzburg-{L}andau equation and motion by mean curvature. {I}. {C}onvergence.
	\emph{J. Geom. Anal.} \textbf{7} (1997), no. 3, 437-475.
	
	\bibitem[Su1]{sun2018local} 
	\textsc{Sun, A.}
	Local Entropy and Generic Multiplicity One Singularities of Mean Curvature Flow of Surfaces.
	arXiv:1810.08114.
	
	\bibitem[Su2]{sun2019entropy} 
	\textsc{Sun, A.}
	Entropy in A Closed Manifold and Partial Regularity of Mean Curvature Flow Limit of Surfaces.
	\emph{J. Geom. Anal.}  (2020), https://doi.org/10.1007/s12220-020-00494-z
	
	\bibitem[T]{Tonegawa03} 
	\textsc{Tonegawa, Y.}
	Integrality of varifolds in the singular limit of reaction-diffusion equations.
	\emph{Hiroshima Math. J.} \textbf{33} (2003), no. 3, 323-341.
	
	\bibitem[Z]{Zhu16} 
	\textsc{Zhu, J. J.}
	On the entropy of closed hypersurfaces and singular self-shrinkers.
	\emph{J. Differential Geom.} \textbf{114} (2020), no. 3, 551-593.
	
	
\end{thebibliography}
\end{document}